\documentclass%[12pt]
{amsart}

\usepackage{amssymb}
\usepackage{graphicx}
\usepackage[all]{xy}
\usepackage{hyperref}

\renewcommand{\setminus}{\smallsetminus}

\theoremstyle{plain}
\newtheorem{thm}{Theorem}[section]
\newtheorem{theorem}[thm]{Theorem}
\newtheorem*{theorem*}{Theorem}
\newtheorem{thmA}{Theorem}

\newtheorem{corollary}[thm]{Corollary}
\newtheorem{lemma}[thm]{Lemma}

\newtheorem*{fact*}{Fact}

\theoremstyle{remark}
\newtheorem*{rem}{Remark}

\theoremstyle{definition}
\newtheorem{definition}[thm]{Definition}

\newtheorem{example}[thm]{Example}

\newcommand{\N}{\mathbb{N}}

\newcommand{\R}{\mathbb{R}}

\newcommand{\into}{\hookrightarrow}
\renewcommand{\k}{\ensuremath \kappa}
\newcommand{\defn}[1]{\emph{#1}}

\DeclareMathOperator{\Int}{Int}
\newcommand{\set}[1]{\left\{#1\right\}}
\newcommand{\setp}[2]{\left\{#1 : #2\right\}}

\newcommand{\thmref}[1]{Theorem~\ref{#1}}
\newcommand{\corref}[1]{Corollary~\ref{#1}}
\newcommand{\lemref}[1]{Lemma~\ref{#1}}

\newcommand{\exampleref}[1]{Example~\ref{#1}}

\newcommand{\dC}{\operatorname{dist}_C}
\newcommand{\dG}{\operatorname{dist}_G}

\title{Closed subsets of a CAT(0) 2-complex are intrinsically CAT(0)}
\author{Russell Ricks}

\date{\today}

\begin{document}

\begin{abstract}
Let $\kappa \le 0$, and let $X$ be a locally-finite CAT(\k) polyhedral $2$-complex $X$, each face with constant curvature $\kappa$.
Let $E$ be a closed, rectifiably-connected subset of $X$ with trivial first singular homology.
We show that $E$, under the induced path metric, is a complete CAT($\kappa$) space.
\end{abstract}

\maketitle

CAT(\k) spaces are a well-studied generalization of nonpositive curvature from Riemannian manifolds to metric spaces.  Many combinatorial constructions, such as $M_{\k}$-polyhedral complexes, involve gluing together pieces of well-understood CAT(\k) spaces to form new ones.  We extend the class of known examples by taking closed subsets of $2$-dimensional polyhedral CAT(\k) spaces and considering the induced path metric; under proper assumptions, these subspaces are also CAT(\k).

In \cite{planar}, we showed that any closed, simply-connected, rectifiably-connected subset of the plane (with a constant curvature $\k \le 0$ metric) is a complete CAT(\k) space when endowed with the induced path metric.  In this paper, we extend this result to similar subspaces of a CAT(\k) polyhedral $2$-complex.  In particular, we prove the following result.

\begin{thmA}[\thmref{main_theorem}]
\label{main thmA}
Let $E$ be a closed, rectifiably-connected subspace of a locally-finite {\rm CAT(\k)} $M_{\k}$-polyhedral $2$-complex $X$, where $\k \le 0$.  If $H_1(E) = 0$ then $E$, under the induced path metric, is a complete {\rm CAT(\k)} space.
\end{thmA}

(Throughout this paper, we assume $\k \le 0$.)

We remark that another potential route to prove Theorem \ref{main thmA} is to solve the isoperimetric inequality for simple closed curves \cite{lytchak-wenger18}.
Unfortunately, it is not clear how to do this; in particular, the behavior of simple closed curves can be rather subtle in a CAT(\k) $2$-complex (see Example \ref{scc example} below, for instance), and there is no general $1$-Lipschitz retraction from the full $2$-complex to a closed subset, even if the subset is contractible.

The paper proceeds as follows:  We first generalize some aspects of the Jordan Curve Theorem to the context of CAT(\k) $2$-complexes (\thmref{curve_theorem}).  This allows us to make assertions about the behavior of curves in $Y$ at the points where they are locally geodesic (\lemref{principal_lemma}).  Finally, we reduce the theorem to the case where $X$ is planar, and use a characterization of angles proved in \cite[Theorem B]{planar}.

The author would like to thank Eric Swenson for suggesting the problem.

\section{Preliminaries}

Write $M_\k^2$ for the plane, equipped with the metric of constant curvature $\k$.  (Throughout this paper, we assume $\k \le 0$.)
%$M_k^2$ is the complete, simply connected, Riemannian surface of constant curvature $\k$.

Let $X$ be a geodesic space, and let $\angle_p^{(\k)} (q,r)$ be the angle at $\overline p$ in the comparison triangle $\triangle(\overline p, \overline q, \overline r)$ in $M_\k^2$ for $\triangle(p, q, r)$.  Suppose $\sigma \colon [0,1] \to X$ and $\tau \colon [0,1] \to X$ are constant-speed geodesic segments emanating from the point
$p \in X$, with $\sigma (1) = q$ and $\tau (1) = r$.  The \defn{Alexandrov angle} between $\sigma$ and $\tau$ is defined as
\[\angle_p (\sigma,\tau) = \lim_{\epsilon \to 0^+} \sup_{0 < t,t' < \epsilon}
\angle_p^{(0)} (\sigma(t),\tau(t')).\] If $X$ is uniquely geodesic, we can define Alexandrov angle at $p$ between $q$ and $r$ to be
\[\angle_p (q,r) = \lim_{\epsilon \to 0^+} \sup_{0 < t,t' < \epsilon}
\angle_p^{(0)} (\sigma(t),\tau(t')).\] 
We will use the following characterization of CAT(\k) spaces \cite[Proposition~II.1.7(4)]{bridson}.

\begin{definition}
For a geodesic metric space $X$, and $\k \le 0$, we say that $X$ is {\rm CAT(\k)} if $\angle_p (q,r) \le \angle_p^{(\k)} (q,r)$ for every triple of distinct points $p,q,r \in X$.
\end{definition}

Every CAT(\k) space (with $\k \le 0$) is a CAT(0) space, and every CAT(0) space is uniquely geodesic.
For more on CAT(\k) spaces, we refer the reader to \cite{bridson}.
%Another good reference is \cite{ballmann}, which gives a treatment of CAT(0) spaces more especially.

\begin{definition}
For distinct points $x$ and $y$ in a uniquely geodesc metric space $Z$, write $[x,y]_Z$ for the geodesic in $Z$ between $x$ and $y$.
\end{definition}

Throughout this paper, we will assume $X$ is a locally-finite CAT(\k) $M_{\k}^2$-polyhedral $2$-complex.
%---with metric $d$.
The technical hypotheses that $X$ be locally finite and that each face of $X$ has %constant
curvature $\k$ ensure that every point $p \in X$ has a \defn{conical neighborhood}---that is, an open neighborhood isometric to a convex open set in the $\k$-cone over the link of $p$, and this isometry maps $p$ to the cone point \cite[Theorem I.7.39]{bridson}.

We will often identify a simple closed curve with its image throughout the paper.

\begin{definition}
Let $\gamma$ be a simple closed curve in $X$.  Define the \defn{interior of $\gamma$} to be
\[\Int \gamma = \setp{x \in X \setminus \gamma}{[\gamma] \neq 0 \in H_1(X \setminus \set{x})}.\]
Note that $\gamma \cup \Int \gamma$ is the intersection of the images of all singular 2-chains with boundary $\gamma$.
In particular, $\gamma \cup \Int \gamma$ is compact.
\end{definition}

We remark that the interior of a simple closed curve can be somewhat subtle, as the following example illustrates.

\begin{figure}[h]
\includegraphics{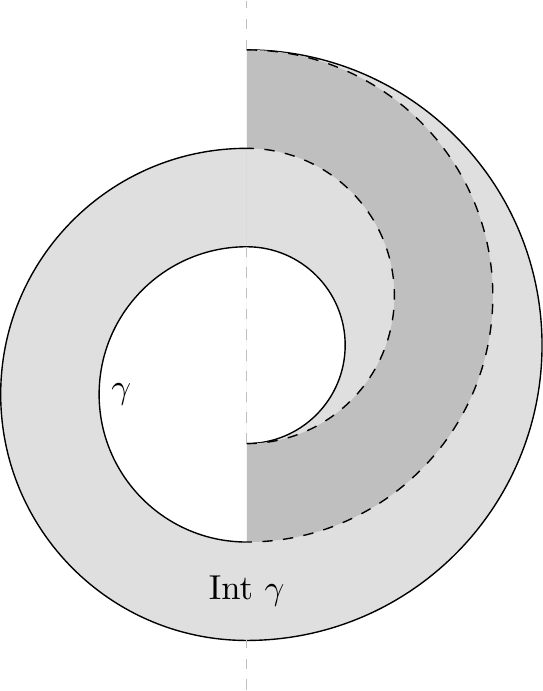}
\caption{The interior of the simple closed curve $\gamma$ of \exampleref{scc example} is not homeomorphic to a disk, nor is it open.
And though $\gamma$ is nullhomologous in $\gamma \cup \Int \gamma$, it is not nullhomotopic (in particular, $\gamma$ does not bound a disk in $\gamma \cup \Int \gamma$).}
\label{fig:example}
\end{figure}

\begin{example}
\label{scc example}
Let $X$ be the CAT(0) Euclidean $2$-complex formed by gluing three flat half-planes $H_1, H_2, H_3 = \setp{(x,y)}{x \ge 0}$, along the boundary edge of each.
Let $\gamma$ be a simple closed curve in $X$ that follows the following pattern (see Figure \ref{fig:example}):
Trace arcs from $(0,-1)$ to $(0,1)$ in $H_1$, then from $(0,1)$ to $(0,-2)$ in $H_2$, then from $(0,-2)$ to $(0,3)$ in $H_3$, then from $(0,3)$ to $(0,-3)$ in $H_1$, then from $(0,-3)$ to $(0,2)$ in $H_2$, then from $(0,2)$ to $(0,-1)$ in $H_3$.
The interior of $\gamma$ is homeomorphic to a punctured torus (not a disk).
It is also not open in $X$.

Notice that replacing $H_1$ by $\pi_1$ in the definition of $\Int \gamma$ defines a strictly larger set in this example---the point $(0,0) \in X$, for instance, lies in $\Int_{\pi_1} \gamma$ but not in $\Int_{H_1} \gamma$.
\end{example}

\section{The Curve Theorem}

The interior of a simple closed curve in $X$ is not open in general, but it is when restricted to the faces (i.e., open 2-cells) of $X$.

\begin{lemma}%[``The Open Face Lemma'']
\label{open_face_lemma}
Let $\gamma$ be a simple closed curve in $X$.  Then the intersection of $\Int \gamma$ with each face of $X$ is open in $X$.
\end{lemma}
\begin{proof}
Suppose $x \in \Int \gamma$ lies on a face $X_f$ of $X$.  Find $\delta > 0$ such that $U := B_X (x, \delta) \subset X_f \setminus \gamma$.  Let $y \in U$; since $\overline{U}$ is homeorphic to the closed disk, we have a deformation retraction of $\overline{U} \setminus \set{y}$ to $\partial U$.  Thus $X \setminus U$ is a deformation retract of $X \setminus \set{y}$ by the Pasting Lemma.  Hence the inclusion $X \setminus U \into X \setminus \{y\}$ is a homotopy equivalence.  So we have ismorphisms $H_1 (X \setminus \set{x}) \to H_1 (X \setminus U) \to H_1 (X \setminus \set{y})$.  Thus $[\gamma] \neq 0 \in H_1 (X \setminus \set{x})$ gives us $[\gamma] \neq 0 \in H_1 (X \setminus \set{y})$, hence $y \in \Int \gamma$.  Therefore $U \subset \Int \gamma$, which proves $X_f \cap \Int \gamma$ is open in $X$.
\end{proof}

For completeness of exposition, we now prove two simple lemmas.

\begin{lemma}%[``The Disk Lemma'']
\label{disk_lemma}
Let $D$ be the closed unit disk in the plane, and let $x \in \partial D$.  For every open neighborhood $U$ of $x$, there is an open neighborhood $V \subset U$ of $x$, a point $z \in V \cap \Int D$, and a deformation retraction $D \setminus \set{z} \to \partial D$ such that the image of $D \setminus V$ lies in $\partial D \setminus V$.
\end{lemma}
\begin{proof}
Let $U$ be an open neighborhood of $x$.  Then there is some $\delta > 0$ such that $V := B (x, \delta) \subset U$.  Pick $z \in V$, and let $f$ be the radial projection of $D \setminus \set{z}$ onto $\partial D$.  By convexity of $V$, $f^{-1}(V) \subset V$.  Thus $f$ is a deformation retraction such that $f(D \setminus V) \subset \partial D \setminus V$.
\end{proof}

\begin{lemma}%[``The Simple Arc Lemma'']
\label{simple_arc_lemma}
Let $\alpha \colon [0,1] \to B$ be a topological embedding into a space $B$ such that $\alpha([0,1]) \setminus (\alpha(0) \cup \alpha(1))$ is open in $B$.  Let $A = \alpha([0,1])$.
Assume there is a continuous map $\sigma \colon S^1 \to B$ of the unit circle to $B$, and let $C = \sigma^{-1}(A)$.  If $\sigma \rvert_{C}$ is a bijection onto $A$ then $[\sigma] \neq 0 \in H_1(B)$.
\end{lemma}
\begin{proof}
Let $f \colon B \to C / \partial C$ be given by $f \rvert_{A} = \left( \sigma \rvert_{C} \right)^{-1}$ and $f(b) = \partial C$ for $b \in B \setminus A$.  Let $h \colon C / \partial C \to S^1$ be a homeomorphism.  Then $h \circ f \circ \sigma$ is a degree $\pm 1$ map $S^1 \to S^1$, hence $(h \circ f)_* [\sigma] \neq 0$.  Therefore $[\sigma] \neq 0$.
\end{proof}

We now extend one important feature of the Jordan Curve Theorem to $X$:  The interior of a simple closed curve $\gamma$ in $X$ accumulates on each point of $\gamma$.
We first prove this result for points of $\gamma$ on the faces of $X$.

\begin{lemma}%[``The Accumulation Lemma'']
\label{accumulation_lemma}
Let $\gamma$ be a simple closed curve in $X$, and let $x$ be a point on $\gamma$ that lies on a face of $X$.  Then $\Int \gamma$ accumulates on $x$.
\end{lemma}
\begin{proof}
Let $r > 0$ be such that $\overline{B_X} (x, r)$ lies in a face of $X$.
Let $\delta \in (0,r)$ be given; we show there is some point of $\Int \gamma$ in $B_X (x,\delta)$.
Note we may assume $\delta$ is small enough that $\gamma \nsubseteq \overline{B_X} (x, \delta)$.
Fix $\delta'$ with $\delta < \delta' < r$.

Let $p$ and $q$ be the endpoints of the maximal arc $\alpha$ of $\gamma$ in $\overline{B_X} (x, \delta)$ on which $x$ lies.  By choice of $r$, the geodesics $[x, p]_X$ and $[x, q]_X$ uniquely extend to geodesics $[x, p']_X$ and $[x, q']_X$, respectively, with $p', q' \in \partial B_X (x, \delta')$.  Note that (up to reparametrization) there are two arcs in $\partial B_X (x, \delta')$ from $q'$ to $p'$; let $\beta_1$ be one and $\beta_2$ the other.  Let $\eta_1$ and $\eta_2$ be the concatenated paths $\eta_i = \alpha * [q, q']_X * \beta_i * [p', p]_X$ (see Figure \ref{fig:acclem}).  Thus $\eta_1$ and $\eta_2$ are simple closed curves in $B_X (x, r)$.

\begin{figure}[h]
\includegraphics{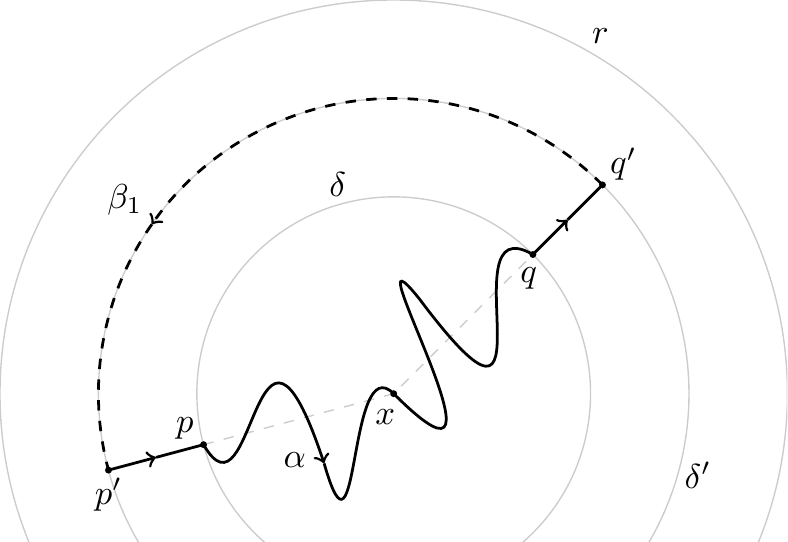}
\caption{The curve $\eta_1 = \alpha * [q, q']_X * \beta_1 * [p', p]_X$.} % in the proof of \lemref{accumulation_lemma}.}
\label{fig:acclem}
\end{figure}

For $i = 1, 2$, let $U_i$ be the bounded component of $B_X (x, r) \setminus \eta_i$ guaranteed by the Jordan Curve Theorem.  Note that $U_1$ and $U_2$ are disjoint, and by Schoenflies, each $\overline{U_i}$ is homeomorphic to the closed unit disk.  By \lemref{disk_lemma}, we may find a nondegenerate subarc $\alpha'$ of $\alpha$ containing $x$, points $z_i \in U_i \cap B_X (x, \delta)$, and deformation retractions of $\overline{U_i} \setminus \set{z_i}$ onto $\eta_i$ mapping $(\gamma \setminus \alpha') \cap B_X (x, r)$ to $\eta_i \setminus \alpha'$.  Pasting these deformation retractions together with the geodesic retraction (projection) of $X$ onto $\overline{B_X} (x, \delta')$, we have a deformation retraction $f \colon X \setminus \set{z_1, z_2} \to \eta_1 \cup \eta_2$ such that $\gamma \cap f^{-1}(\alpha') = \alpha'$.  Thus
\[\xymatrix{
H_1 (X \setminus \set{z_1, z_2}) \ar[r]^-{f_*} & H_1 (\eta_1 \cup \eta_2)
}\]
is an isomorphism that maps $[\gamma] \mapsto f_*[\gamma] \neq 0$ by \lemref{simple_arc_lemma}.  Hence $[\gamma] \neq 0 \in H_1 (X \setminus \set{z_1, z_2})$.  By Mayer-Vietoris (recall that $X$ is contractible) we have an isomorphism
\[\xymatrix{
H_1 (X \setminus \set{z_1, z_2}) \ar[r] & H_1 (X \setminus \set{z_1}) \oplus H_1 (X \setminus \set{z_2})
}\]
induced by inclusions, and thus $[\gamma] \neq 0$ in at least one of $H_1 (X \setminus \set{z_1})$ and $H_1 (X \setminus \set{z_2})$.  Therefore, either $z_1$ or $z_2$ is an element of $\Int \gamma$.
\end{proof}

We now drop the hypothesis that $x$ lie on a face of $X$.

\begin{lemma}\label{curve_theorem_1}
Let $\gamma$ be a simple closed curve in $X$.
Then $\Int \gamma$ accumulates on every point of $\gamma$.
In fact, the points of $\Int \gamma$ that lie on faces of $X$ so accumulate.
\end{lemma}
\begin{proof}
Let $x \in \gamma$, and let $\delta > 0$ be given.  We will show that there is some point of $\Int \gamma$ in $B_X (x, \delta)$.  If some $y \in \gamma \cap B_X (x, \delta)$ lies on a face of $X$, \lemref{accumulation_lemma} gives us a point of $\Int \gamma$ in $B_X (x, \delta)$.  So suppose $\gamma \cap B_X (x, \delta)$ lies in the $1$-skeleton of $X$.  Then there is a $y \in \gamma \cap B_X (x, \delta)$ that lies on an edge of $X$, and a $\delta' \in (0, d(x,y))$ such that $B_X (y, \delta')$ contains no vertices of $X$ but $\gamma \cap B_X (y, \delta')$ is a geodesic segment along the edge.

For each face $X_f$ of $X$ that touches the edge, choose a point $y_f \in X_f$ such that $\frac{1}{2} \delta' < d(y_f, y) < \delta'$.  Let $P$ be the set of all points $y_f$.  Since $\overline{B}_X (y, \delta') \setminus P$ deformation retracts onto $\partial B_X (y, \delta') \cup \left( \gamma \cap B_X (y, \delta') \right)$, by \lemref{simple_arc_lemma} we have $[\gamma] \neq 0 \in H_1 (X \setminus P)$.
Now $P$ is discrete and closed by construction, and
$X$ is contractible and locally contractible, so
the inclusions $X \setminus P \to X \setminus \set{p}$ induce an isomorphism
\[\xymatrix{
H_* (X \setminus P) \ar[r]^-{\cong} & \bigoplus_{p \in P} H_* (X \setminus \set{p}).
}\]
Thus one of the points $y_f$ must lie in $\Int \gamma$.
\end{proof}

Our next goal is to prove a replacement property for openness of the interior $\Int \gamma$ of a simple closed curve $\gamma$ in $X$, to use when not on a face of $X$.
We prove $\Int \gamma$ has the geodesic extension property, in a local sense.

\begin{lemma}%[``The Geodesic Extension Lemma'']
\label{geodesic_extension_lemma}
Let $\gamma$ be a simple closed curve in $X$.  Then $\Int \gamma$ locally extends geodesics; that is, every geodesic $[p,q]_X \subset \Int \gamma$ extends to a geodesic $[p,q']_X \subset \Int \gamma$ with $[p,q]_X \subset [p,q']_X$.
\end{lemma}

\begin{proof}
Let $[p,q]_X \subset \Int \gamma$ be a geodesic in $X$.
Find $\epsilon > 0$ such that $D := B_X (q, \epsilon)$ is a conical neighborhood of $q$, and $\gamma \subset X \setminus D$.
Now $\overline{D} \setminus \set{q}$ deformation retracts onto $\partial D$, and $X$ deformation retracts onto $\overline{D}$ via geodesic projection, so pasting together gives a deformation retraction $f \colon X \setminus \set{q} \to \partial D$.

Let $C = \setp{y \in \partial D}{\angle_q (p, y) < \pi}$.
Since $X$ is locally finite, the topological boundary $P = \partial C$ in $\partial D$ is finite, so $\partial D \setminus P$ is a disjoint union of metric graphs.
By the link condition on $X$ \cite[p. 206]{bridson}, the component $C$ of $\partial D \setminus P$ is contractible, and there is a lower bound on the length of circles in $\partial D$; hence there exists a finite set $Q \subset \partial D \setminus C$ such that $\partial D \setminus Q$ is a disjoint union of trees.
The long exact sequence
\[\xymatrix{
\dotso \ar[r] & H_k (\partial D \setminus Q) \ar[r] & H_k (\partial D) \ar[r]^-{\varphi_k} & H_k (\partial D, \partial D \setminus Q) \ar[r] & \dotso
}\]
then has $H_k (\partial D \setminus Q) = 0$ for $k > 0$, so in particular the map %(induced by chain-level projection)
\[\xymatrix{H_1 (\partial D) \ar[r]^-{\varphi_1} & H_1 (\partial D, \partial D \setminus Q)}\]
is injective.
Let $\set{V_y}_{y \in Q}$ be a collection of pairwise-disjoint open sets in $\partial D$ such that each $y \in V_y$.
Then we have a commutative diagram
\[\xymatrix{
H_1 (\coprod_{y \in Q} V_y, (\coprod_{y \in Q} V_y) \setminus Q) \ar[r] \ar[d]
& \bigoplus_{y \in Q} H_1 (V_y, V_y \setminus \set{y}) \ar[d] \\
H_1 (\partial D, \partial D \setminus Q) \ar[r]^-{\Psi}
& \bigoplus_{y \in Q} H_1 (\partial D, \partial D \setminus \set{y}),
}\]
where the vertical maps are induced by inclusion, the top map is the canonical isomorphism, and the bottom map is induced by the inclusions $\partial D \setminus Q \to \partial D \setminus \set{y}$. %$\xymatrix@1{\partial D \setminus Q \ar[r] & \partial D \setminus \set{y}}$.
By excision, the vertical maps are isomorphisms, so $\Psi$ is an isomorphism.
Let $\Phi$ be the composition
\[\xymatrix{
H_1 (\partial D) \ar[r]^-{\varphi_1}
& H_1 (\partial D, \partial D \setminus Q) \ar[r]^-{\Psi}
& \bigoplus_{y \in Q} H_1 (\partial D, \partial D \setminus \set{y}).
}\]
Since the maps
$f_*$ %$f_* \colon H_1 (X \setminus \set{q}) \to H_1 (\partial D)$
and
$\Phi$ %$\Phi \colon H_1 (\partial D) \to \bigoplus_{y \in Q} H_1 (\partial D, \partial D \setminus \set{y})$
%\[\xymatrix{
%H_1 (X \setminus \set{q}) \ar[r]^-{f_*}
%& H_1 (\partial D) \ar[r]^-{\Phi}
%& \bigoplus_{y \in Q} H_1 (\partial D, \partial D \setminus \set{y})
%}\]
are both injective, and $[\gamma] \neq 0 \in H_1 (\partial D)$, there must be some $q' \in Q$ such that $(\phi \circ f_*) ([\gamma]) \neq 0$, where
\[\xymatrix{
\phi \colon H_1 (\partial D) \ar[r]
& H_1 (\partial D, \partial D \setminus \set{q'})
}\]
is the associated component function of $\Phi$.
(Hence by naturality of the long exact sequence, $\phi$ is the map in the long exact sequence of the pair $(\partial D, \partial D \setminus \set{q'})$.)

Since $q' \in \partial D \setminus C$, we have $\angle_q (p,q') = \pi$, and therefore $[q,q']_X$ extends the geodesic segment $[p,q]_X$ to $[p,q']_X$.
Thus it remains only to show $[q,q']_X \subset \Int \gamma$.

Define $W \subset D$ and $U \subset \partial D$ as follows.
If $q'$ lies on a face of $X$, then let $F$ be the open face containing $q'$, and let $W = F \cap D$ and $U = F \cap \partial D$.
On the other hand, if $q'$ lies on an edge of $X$, then let $G = [q,q']_X$, let $F$ be the union of all open faces containing $q'$ in their closure, and let $W = (F \cup G) \cap D$ and $U = (F \cup G) \cap \partial D$.
Now let $x \in [q, q']_X \setminus \set{q, q'}$ be arbitrary.
By construction, $W$ is a conical neighborhood of $q'$ in $X$, so we have a deformation retraction $g \colon X \setminus \set{x} \to \partial W$.
Also, $\partial W \setminus \set{q'}$ deformation retracts onto $\partial W \setminus U$, which deformation retracts onto $\set{q}$, being a cone over $q$.
Thus $\partial W \setminus \set{q'}$ is contractible, so the long exact sequence gives us an isomorphism
\[\xymatrix{
H_1 (\partial W) \ar[r]^-{\theta} & H_1 (\partial W, \partial W \setminus \set{q'}).
}\]
Let $\psi$ be the map that makes the following diagram commute, where the other two maps are the isomorphisms (by excision) induced by inclusion
\[\xymatrix{
&H_1 (U, U \setminus \set{q'})
\ar[dl] \ar[dr] \\
H_1 (\partial W, \partial W \setminus \set{q'}) \ar[rr]^{\psi}
&& H_1 (\partial D, \partial D \setminus \set{q'}).}\]
Thus we have isomorphisms
\[\xymatrix{H_1 (X \setminus \set{x}) \ar[r]^{g_*}
& H_1 (\partial W) \ar[d]^{\theta} \\
& H_1 (\partial W, \partial W \setminus \set{q'}) \ar[r]^{\psi}
& H_1 (\partial D, \partial D \setminus \set{q'}).
}\]
Let $\xymatrix@1{X \setminus D \ar[r]^-{i} & X \setminus \set{q}}$ and $\xymatrix@1{X \setminus D \ar[r]^-{j} & X \setminus \set{x}}$ be inclusion.
Then the square
\[\xymatrix{H_1 (X \setminus D) \ar[r]^-{j_*} \ar[dd]_{i_*} %\ar[ddr]^{(f \rvert_{X \setminus D})_*}
& H_1 (X \setminus \set{x}) \ar[r]^{g_*}
 & H_1 (\partial W) \ar[d]^{\theta} \\
&& H_1 (\partial W, \partial W \setminus \set{q'}) \ar[d]^{\psi} \\
H_1 (X \setminus \set{q}) \ar[r]^-{f_*}
& H_1 (\partial D) \ar[r]^-{\phi}
& H_1 (\partial D, \partial D \setminus \set{q'}),
}\]
commutes, and thus $q \in \Int \gamma$ implies $x \in \Int \gamma$ for all $x \in [q, q']_X \setminus \set{q, q'}$.
By compactness of $\gamma \cup \Int \gamma$, we obtain $q' \in \Int \gamma$, and the theorem is proved.
\end{proof}

Thus we can extend another feature of the Jordan Curve Theorem to $X$.

\begin{corollary}
\label{curve_theorem_2}
Let $\gamma$ be a simple closed curve in $X$.  Then $\Int \gamma$ extends geodesics to $\gamma$---that is, every geodesic $[p, q]_X \subset \Int \gamma$ extends to a geodesic $[p, q']_X \subset \gamma \cup \Int \gamma$ with $[p, q]_X \subset [p, q']_X$, $q' \in \gamma$, and $[p, q']_X \setminus \set{q'} \subset \Int \gamma$.
\end{corollary}
\begin{proof}
Let
\(a' = \sup \setp{d (p, q')}{q' \in X \text{ such that } [p, q]_X \subset [p, q']_X \subset \gamma \cup \Int \gamma}\).
By compactness of $\gamma \cup \Int \gamma$, there is some $q' \in \gamma \cup \Int \gamma$ such that $d_X (p, q') = a'$ and $[p, q]_X \subset [p, q']_X \subset \gamma \cup \Int \gamma$.
By \lemref{geodesic_extension_lemma}, $q' \notin \Int \gamma$; hence $q' \in \gamma$.

Now, it is conceivable that $q'$ is not the only point of $\gamma$ on $[p,q']_X$.
If so, by compactness of $\gamma$ we may find the closest point $q''$ to $p$ on $[p, q']_X \cap \gamma$.
Then the conclusion of the lemma holds with $q''$ in place of $q'$.
\end{proof}

%The following theorem summarizes our extensions of certain features of the Jordan Curve Theorem to $X$.
Combining \lemref{curve_theorem_1} and \corref{curve_theorem_2}, we obtain the following theorem.

\begin{theorem}[``The Curve Theorem'']\label{curve_theorem}
Let $\gamma$ be a simple closed curve in $X$.
Then $\Int \gamma$ accumulates on every point of $\gamma$.
Moreover, $\Int \gamma$ extends geodesics to $\gamma$.
\end{theorem}

\section{Detecting intrinsic geodesics}

Let $E$ be a fixed closed, rectifiably-connected subspace of $X$ with the induced subspace metric (also written $d$), such that $H_1(E) = 0$.  Let $Y$ be the space $E$, endowed with the induced path metric, $d_Y$.

\begin{lemma}\label{complete and geodesic}
Let $Z$ be a complete, rectifiably-connected metric space.
Then $Z$, with the induced path metric, is complete and geodesic.
\end{lemma}
\begin{proof}
This is proved in Lemma 2.1 and Corollary 2.2 of \cite{planar}.
\end{proof}

Of course this implies $Y$ is complete and geodesic.

\begin{lemma}\label{prop1}
If the geodesic $[x,y]_X$ in $X$ between $x$ and $y$ satisfies $[x,y]_X \subset E$, then $[x,y]_X$ is the unique geodesic \underline{in $Y$} between $x$ and $y$.
\end{lemma}
\begin{proof}
This result is immediate from the definitions.
\end{proof}

%\begin{lemma}\label{prop2}
%If $[x,y]_X$ lies in $E$, and $\alpha$ is an arc in $E$ with $\alpha \cap [x,y]_X = \{x, y\}$, then $\alpha$ is not a geodesic in $Y$.
%\end{lemma}
%\begin{proof}
%This result is immediate from \lemref{prop1}.
%\end{proof}

\begin{lemma}\label{prop3}
Let $\gamma$ be a simple closed curve in $E$, and $\alpha$ be a subarc of $\gamma$ that is geodesic in $Y$.
If $x,y \in \alpha$ such that $[x,y]_X$ lies in $\gamma \cup \Int \gamma$, then $[x,y]_X \subset \alpha$.
\end{lemma}
\begin{proof}
Since $H_1(E) = 0$, we know $\Int \gamma \subset E$.
%Thus $\gamma \cup \Int \gamma \subset E$.
Apply \lemref{prop1}. %\lemref{prop2}.
\end{proof}

\begin{lemma}%[``The Principal Lemma'']
\label{principal_lemma}
Let $C$ be a closed convex subspace of $X$ and $\gamma$ be a simple closed curve in $E$.  Define $\dC \colon X \to \R$ by $\dC(p) = d(p, C)$.  Let $p$ be a point on $\gamma$ where $\dC$ attains its maximum on $\gamma$ but is not locally constant on $\gamma$.  Then $\gamma$ is not locally geodesic in $Y$ at the point $p$.
\end{lemma}

\begin{proof}
Suppose, by way of contradiction, that $V$ is a conical neighborhood of $p$ in $X$ such that $\gamma \cap V$ is geodesic.
By \lemref{curve_theorem_1}, there is a sequence $p_k \to p$ in $X$ such that each $p_k$ lies on a face of $X$ and $p_k \in \Int \gamma$.
For each $p_k$, let $r_k \in C$ be the closest point to $p_k$ in $C$, and let $x'_k, y'_k$ be points of $X \setminus \set{p_k}$ such that the Alexandrov angle in $X$ satisfies $\angle_{p_k}^X (x'_k,r_k) = \frac{\pi}{2} = \angle_{p_k}^X (y'_k,r_k)$ and $\angle_{p_k}^X (x'_k,y'_k) = \pi$.
By \lemref{open_face_lemma}, we may assume $x'_k, y'_k$ are near enough to $p_k$ that the geodesic $[x'_k,y'_k]_X$ in $X$ satisfies $[x'_k,y'_k]_X \subset \Int \gamma$.
By \corref{curve_theorem_2}, we can extend $[x'_k,y'_k]_X$ to a geodesic $[x_k,y_k]_X$ in $X$ such that $[x_k,y_k]_X \subset \gamma \cup \Int \gamma$ and $x_k,y_k \in \gamma$.
Taking the limit of the geodesics $[x_k,y_k]_X$ in $X$ as $k \to \infty$, we obtain a (possibly degenerate) geodesic $[x,y]_X$ in $X$; by compactness, $[x,y]_X \subset \gamma \cup \Int \gamma$ and $x,y \in \gamma$.
Since each $p_k \in [x_k,y_k]_X$, we also see that $p \in [x,y]_X$.

We consider three cases for $x$ and $y$.
Case 1:  $x = p = y$.
Then eventually $x_k, y_k \in V \cap \gamma$, contradicting \lemref{prop3}.
Case 2:  $x \neq p$ and $y \neq p$.
Then $\gamma$ locally follows the geodesic $[x,y]_X$ in $X$.
But now $p$ cannot be a local maximum for $\dC$ without forcing $\dC$ to be locally constant at $p$ because $\dC$ is convex in the CAT(0) space $X$.
So Cases 1 and 2 are impossible.

Finally, Case 3:  Exactly one of $x,y$ equals $p$.
We may assume $x \neq p$ and $y = p$.
Choose $w,q \in [x,p]_X \setminus \set{x,p}$ such that the geodesic $[w,p]_X$ in $X$ extends the geodesic $[w,q]_X$ in $X$.
By \lemref{curve_theorem_1}, there is a sequence $q_k \to q$ in $X$ such that each $q_k$ lies on a face of $X$ and $q_k \in \Int \gamma$.
Note that $[w,q_k]_X \setminus \set{w} \subset \Int \gamma$ for all sufficiently large $k \in \N$ because $V$ is a conical neighborhood of $p$ in $X$.
By \corref{curve_theorem_2}, we may extend each geodesic $[w,q_k]_X$ in $X$ to a geodesic $[w,z_k]_X$ in $X$ such that $z_k \in \gamma$ and $[w,z_k]_X \setminus \set{w,z_k} \subset \Int \gamma$.
Passing to a subsequence if necessary, we may assume $(z_k)$ converges in $X$ to some point $z \in \gamma$.
Notice $p \in [x,z]_X$ by simpleness of $\gamma$.
Since $[x,z]_X \subset \gamma \cup \Int \gamma$ by compactness, if $p \neq z$ then $\gamma$ is locally geodesic in $X$ at $p$, and we obtain a contradiction as in Case 2 above; on the other hand, if $p = z$ then eventually $z_k,w \in V \cap \gamma$, contradicting \lemref{prop3} as in Case 1 above.
Thus in every case, we obtain a contradiction.
The statement of the lemma follows.
\end{proof}

\begin{corollary}
$Y$ is uniquely geodesic.
\end{corollary}
\begin{proof}
Suppose, by way of contradiction, that $\sigma$ and $\tau$ are distinct geodesics in $Y$ from $x \in Y$ to $y \in Y$.  We may assume $\sigma$ and $\tau$ intersect only at $x$ and $y$.  Put $C = [x,y]_X$ and $\gamma = \sigma \cup \tau$.  By \lemref{principal_lemma}, one of $\sigma$ and $\tau$ is not locally geodesic at some point (away from $x$ and $y$), contradicting our hypothesis on $\sigma$ and $\tau$.  Therefore, $Y$ is uniquely geodesic.
\end{proof}

\section{Bootstrapping the planar case}

In the proof of \cite[Theorem B]{planar}, the essential construction is that of two unique (though not necessarily distinct) geodesic segments we call ``limit segments,'' based at a given vertex, $p$, of a simple geodesic triangle $T$, with two important features:
\begin{enumerate}
\item
The limit segments extend through $T \cup \Int T$ to the opposite side of $T$.
\item
The angle between the limit segments equals the Alexandrov angle of the triangle at $p$ (this is the content of \cite[Theorem B]{planar}).
\end{enumerate}
We will follow the same general outline, creating limit segments and proving that the angle between them equals the Alexandrov angle of the triangle at the vertex.  The first part of the proof is to reduce to the case where \cite[Theorem B]{planar} applies.

We begin with some terminology.

\begin{definition}
A geodesic triangle $T$ in $Y$ is called \defn{simple} if it is a simple closed curve in $Y$ (or, equivalently, in $X$).
\end{definition}

\begin{definition}
Let $p$ and $q$ be distinct points in $Y$, and let $\sigma \colon [0,1] \to Y$ be a constant-speed geodesic with $\sigma(0) = p$ and $\sigma(1) = q$.  Let $\epsilon > 0$ be small enough that $B_X (p, 2\epsilon)$ is a conical neighborhood of $p$ in $X$.  For $t \in (0,1]$, let $R_q^{\epsilon,t}$ be the geodesic from $p$ through $\sigma(t)$ in $X$ of length $\epsilon$.  If the geodesics $R_q^{\epsilon,t}$ limit uniformly onto a geodesic $R_q^{\epsilon}$ in $X$, we call $R_q^{\epsilon}$ the \defn{limit segment} of $\sigma$ at $p$ (of length $\epsilon$).
\end{definition}

Since we will be talking about angles in both $X$ and $Y$, we will distinguish between them by placing the space as a superscript, as follows.

\begin{definition}
Let $Z$ be a metric space.  We will write $\angle_p^Z (q, r)$ for the Alexandrov angle $\angle_p (q, r)$ in $Z$ and $\angle_p^{(\k), Z} (q, r)$ for the $M_\k^2$-comparison angle $\angle_p^{(\k)} (q, r)$ in $Z$.
\end{definition}

We restate \cite[Theorem B]{planar} in the current context.

\begin{theorem}
[Theorem B of \cite{planar}]
\label{previous_paper_angles_theorem}
Assume $X = M_\k^2$ and $E$ is simply connected.
Let $T$ be a simple geodesic triangle in $Y$ with vertices $p$, $q$, and $r$.
For all sufficiently small $\epsilon > 0$, both limit segments $R_q^\epsilon$ and $R_r^\epsilon$ of $T$ exist at $p$, both $R_q^\epsilon$ and $R_r^\epsilon$ lie completely in $T \cup \Int T$, and
\[\angle_p^Y (q, r) = \angle_p^X (R_q^\epsilon, R_r^\epsilon).\]
\end{theorem}

We would like to use this theorem to prove the general case, but first we have to properly transfer the setting.

\begin{lemma}\label{convexity_lemma}
Let $C$ be a closed, convex subspace of $X$.
Let $A$ be a path-connected metric subspace of $E$, and let $B$ be the union of all geodesics $[x,y]_Y$ in $Y$ between all points $x,y \in A$, taken as a metric subspace of $E$.
If $A \subset C$, then $B \subset C$. % and $B$ is path connected.
\end{lemma}

\begin{proof}
First suppose, by way of contradiction, that $A \subset C$ but there is a point $z \in B \setminus C$.
By definition of $B$, we know $z$ lies on a geodesic $[x,y]_Y$ in $Y$ for some $x,y \in A$.
We may assume that $[x,y]_Y \cap C = \set{x,y}$.
Since $A$ is path connected, we may find a path in $A \subset C \cap E$ joining $x$ and $y$; we may of course assume this path is injective.
Concatenating this path with the geodesic $[x,y]_Y$, we obtain a simple closed curve in $E$.
By \lemref{principal_lemma}, $[x,y]_Y$ must stay in $C$ or it cannot be a geodesic in $Y$; this contradicts our choice of $[x,y]_Y$.
Therefore, we must have $B \subset C$.
%Note that since $A$ is path connected, $B$ is, too, by construction.
\end{proof}

\begin{lemma}\label{disks_are_self_contained}
Let $D$ be a closed, convex subspace of $X$ which isometrically embeds in $M_\k^2$.
Let $T$ be a simple geodesic triangle in $Y$ with vertices $p,q,r \in D$.
Then $T \cup \Int T \subset D \cap E$ and is compact, simply connected, and rectifiably connected as a subspace of $D$, and convex as a subspace of $Y$.
%under the metric $d_Y$.
\end{lemma}

\begin{proof}
Let $C$ be the convex hull of $\set{p,q,r}$ in $X$; note that $C$ is closed.  By \lemref{principal_lemma}, $T$ must lie in $C$.  Now $C$, as a convex subspace of a CAT(0) space, is contractible.  Therefore, there is a singular 2-chain in $C$ with boundary $T$; hence $T \cup \Int T$, being the intersection of %the images of 
all singular 2-chains with boundary $T$, must lie in $C$.  Thus $T \cup \Int T \subset C \subset D$.  Let $\varphi \colon D \to M_\k^2$ be an isometric embedding. %from the lemma hypothesis.
Note that $\varphi(C)$ is the convex hull of $\set{\varphi(p), \varphi(q), \varphi(r)}$ in $M_\k^2$, and $\varphi(\Int T) = \Int \varphi(T)$ is the interior of $\varphi(T)$ guaranteed by the Jordan Curve Theorem in $M_\k^2$.
In particular, $T \cup \Int T$, as a metric subspace of $X$, is topologically a closed disk by Schoenflies.

Let $F_0 = T \cup \Int T$ as a metric subspace of $X$.
Since $H_1(E) = 0$, we know $\Int T \subset E$ and therefore $F_0 \subset E$.
Let $Z_0 = F_0$ as a metric subspace of $Y$.
%Our next goal is to show $Z_0$ is a convex subset of $Y$.
Let $Z_1$ be the union of all geodesics $[x,y]_Y$ in $Y$ between all points $x,y \in Z_0$, with $Z_1$ taken as a metric subspace of $Y$.
Let $F_1 = Z_1$ as a metric subspace of $X$.
Since $F_0$ is topologically a closed disk, $F_0$ is path connected.
Thus $F_1 \subset C$ by \lemref{convexity_lemma}.

But now $F_0 \subset F_1 \subset C \subset D$.
Since $C$ is the convex hull of the vertices of $T$ in a disk of $M_\k^2$, no geodesic $[x,y]_Y$ in $Y$ can stay in $C$ and yet leave $F_0$ without violating the hypothesis that the sides of $T$ are geodesics in $Y$.
Thus $F_1 = F_0$, i.e.~ $Z_0$ is convex.
This implies $F_0$ is rectifiably connected, and the lemma is proved.
\end{proof}

The next lemma gives some surprising teeth to \lemref{disks_are_self_contained}.

\begin{lemma}\label{small_simple_triangles_are_planar}
Let $p \in Y$, and let $V$ be a conical neighborhood of $p$ in $X$.
Every simple geodesic triangle in $Y$ with one vertex $p$ and other two vertices $q,r \in V$ lies inside a closed, convex subspace $C$ of $X$ that isometrically embeds in $M_\k^2$.
\end{lemma}

\begin{proof}
Let $q,r$ be points of $E \cap V$ such that the geodesic triangle $T$ in $Y$ with vertices $p,q,r$ is simple.
Let $C$ be the convex hull of $\set{p,q,r}$ in $X$; note that $C$ is closed and $C \subset V$.
By \lemref{principal_lemma}, $T$ must lie in $C$.
Now if $\angle_p^X (q,r) = \pi$, then $C$ is the geodesic $[q,r]_X$ in $X$; however, this means $T \subset C$ cannot be simple, contradicting our hypothesis.
Thus $\angle_p^X (q,r) < \pi$.

Let $\rho$ be radial projection in $X \setminus \set{p}$ to the link of $p$ in $X$, and let $G$ be the geodesic in the link between $\rho(q)$ and $\rho(r)$.
Let $D = (\rho^{-1} (G) \cap V) \cup \set{p}$.
Since $V$ is a conical neighborhood about $p$, it follows that $D$ is convex.
Thus $C \subset D$; since $D$ isometrically embeds in $M_\k^2$, so does $C$.
\end{proof}

\begin{corollary}\label{small_simple_limit_segments}
Let $p \in Y$, and let $V$ be a conical neighborhood of $p$ in $X$.
Let $T$ be a simple geodesic triangle in $Y$ with one vertex $p$ and other two vertices $q,r \in V$.
There exists $\epsilon > 0$ such that both limit segments $R_q^\epsilon$ and $R_r^\epsilon$ of $T$ exist at $p$, both $R_q^\epsilon$ and $R_r^\epsilon$ lie completely in $T \cup \Int T$, and
\[\angle_p^Y (q, r) = \angle_p^X (R_q^\epsilon, R_r^\epsilon).\]
\end{corollary}

%\begin{proof}
%By
%\lemref{small_simple_triangles_are_planar},
%there is a closed, convex subspace $C \supset T$ of $X$ that isometrically embeds in $M_\k^2$.  By
%\lemref{disks_are_self_contained},
%we have $T \cup \Int T \subset D \cap E$, while $T \cup \Int T$ is compact, simply connected, and rectifiably connected as a subspace of $D$, and convex as a subspace of $Y$.  And so by
%\thmref{previous_paper_angles_theorem}.
%there exists $\epsilon > 0$ such that both limit segments $R_q^\epsilon$ and $R_r^\epsilon$ of $T$ exist at $p$, both $R_q^\epsilon$ and $R_r^\epsilon$ lie completely in $T \cup \Int T$, and
%\[\angle_p^Y (q, r) = \angle_p^X (R_q^\epsilon, R_r^\epsilon).\]
%\end{proof}

\begin{proof}
Apply
\lemref{small_simple_triangles_are_planar},
\lemref{disks_are_self_contained},
and
\thmref{previous_paper_angles_theorem}.
\end{proof}

We began this section by mentioning two important features that held in the planar case, which we wanted to extend to the general case.
We can now prove existence of limit segments in the general case, along with Feature 2.

\begin{theorem}\label{limit_segment_theorem}
Let $T$ be a simple geodesic triangle in $Y$ with vertices $p$, $q$, and $r$.
For all $\epsilon > 0$ small enough, both limit segments $R_q^\epsilon$ and $R_r^\epsilon$ of $T$ exist at $p$, and
\[\angle_p^Y (q, r) = \angle_p^X (R_q^\epsilon, R_r^\epsilon).\]
\end{theorem}
\begin{proof}
Let $V$ be a conical neighborhood about $p$ in $X$.
Since the existence of limit segments of $T$ at $p$ is local, and $Y$ uniquely geodesic, we may assume $q,r \in V$.
The theorem follows from \corref{small_simple_limit_segments}.
\end{proof}

\section{Extending limit segments}

We now turn to proving Feature 1 of the limit segments in the general case.
We begin with a result about variations of geodesics in $Y$.

\begin{lemma}\label{continuous_variation}
Let $p_k \to p$ and $q_k \to q$ in $Y$.
For each $k \in \N$ let $\sigma_k \colon [0,1] \to Y$ be the constant-speed parametrization of the geodesic in $Y$ from $p_k$ to $q_k$.
Then the maps $\sigma_k$ converge uniformly \underline{in $X$} to the constant-speed parametrization $\sigma \colon [0,1] \to Y$ of the geodesic in $Y$ from $p$ to $q$.
\end{lemma}
\begin{proof}
Let $\iota \colon Y \to E$ be the setwise identity map, $\iota(y) = y \in E$ for all $y \in Y$. %, and note that $\iota$ is $1$-Lipschitz.
For each $k \in \N$ let $\sigma_k^E = \iota \circ \sigma_k$, and note that $\sigma_k^E \colon [0,1] \to E$ is $d_Y(p_k,q_k)$-Lipschitz.
Since $p_k \to p$ and $q_k \to q$ in $Y$, we know $d_Y(p_k,q_k) \to d_Y(p,q)$.
Since $E$ is proper, by the Arzel\`a-Ascoli Theorem the maps $\sigma_k^E$ have a uniform limit $\sigma^E \colon [0,1] \to E$, which is $d_Y(p,q)$-Lipschitz.
By uniqueness of geodesics in $Y$, we must have $\sigma^E = \iota \circ \sigma$.
\end{proof}

\begin{rem}
If the uniform convergence in \lemref{continuous_variation} was in $Y$ instead of in $X$, then we could say geodesics in $Y$ vary continuously in their endpoints.
%(The difficulty in proving geodesics in $Y$ vary continuously in their endpoints is actually in getting from uniform convergence in $X$ to pointwise convergence in $Y$.)
\end{rem}

\begin{corollary}\label{geodesic_variations}
Let $\iota \colon Y \to E$ be the setwise identity map, $\iota(y) = y \in E$ for all $y \in Y$.
Let $\tau_0, \tau_1 \colon [0,1] \to Y$ be two continuous paths in $Y$, and for each $s \in [0,1]$ let $\sigma_s \colon [0,1] \to Y$ be constant-speed parametrization of the (possibly degenerate) geodesic $[\tau_0(s), \tau_1(s)]_Y$ in $Y$ from $\tau_0(s)$ to $\tau_1(s)$.
Then the map $f \colon [0,1] \times [0,1] \to E$ defined by setting $f(s,t) = \iota(\sigma_s(t))$ for all $s,t \in [0,1]$ is continuous in $X$.
\end{corollary}

\noindent Thus a variation of geodesics in $Y$ is a homotopy in $E \subset X$ (see Figure \ref{fig:variation}).

\begin{figure}[h]
\includegraphics{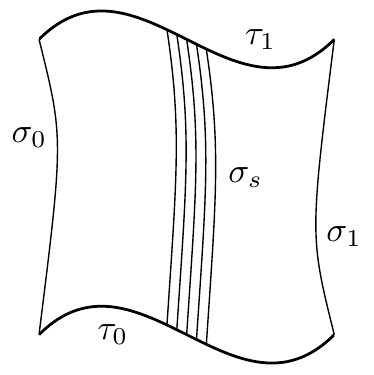}
\caption{A variation of geodesics in $Y$ gives a homotopy in $E$ (as a subset of $X$), by \corref{geodesic_variations}.}
\label{fig:variation}
\end{figure}

This homotopy allows us to compare interiors of two related triangles (a similar statement for generic simple closed curves in $Y$ would be false).

\begin{lemma}\label{subtriangle_interiors}
Let $T$ be a simple geodesic triangle in $Y$ with vertices $p,q,r$.
Let $\widehat{T}$ be a simple geodesic triangle in $Y$ with vertices $p,q',r'$, where $q' \in [p,q]_Y$ and $r' \in [p,r]_Y$.
There is some $\epsilon > 0$ such that $\Int T \cap B_X (p, \epsilon) = \Int \widehat{T} \cap B_X (p, \epsilon)$.
\end{lemma}

\begin{proof}
Let $P \subset Y$ be the broken-geodesic path $P = [q',q]_Y \cup [q,r]_Y \cup [r,r']_Y$ in $Y$, and $Q \subset Y$ the quadrilateral $Q = P \cup [r',q']_Y$.
Let $G$ be the geodesic $[q,r]_X$ between $q$ and $r$ in $X$.
By \lemref{principal_lemma}, $p$ globally maximizes $\dG$ (the distance in $X$ to $G$) among points of $T$.
Hence by compactness, there is some $\epsilon > 0$ such that $\dG(x) \le \dG(p) - \epsilon$ for all $x \in P$.
Let $C \subset X$ be the closed, convex subspace $C = \setp{x \in X}{d(x,G) \le \epsilon}$.
Then $P \subset C$, and by \lemref{convexity_lemma}, $Q \subset C$.
\corref{geodesic_variations} guarantees a nulhomotopy of $Q$ in $E$.
Applying \lemref{convexity_lemma} again, we see that this nulhomotopy lies completely in $C$.
Therefore, $T$ and $\widehat{T}$ are homotopic in $E$ by a homotopy that is constant on $X \setminus C \supset B_X (p, \epsilon)$.
In particular, $\Int T \cap B_X (p, \epsilon) = \Int \widehat{T} \cap B_X (p, \epsilon)$.
\end{proof}

\begin{lemma}\label{containing_limit_segments}
Let $T$ be a simple geodesic triangle in $Y$ with vertices $p,q,r$.
For small enough $\epsilon > 0$, both limit segments $R_q^\epsilon$ and $R_r^\epsilon$ of $T$ at $p$ lie completely in $T \cup \Int T$.
\end{lemma}

\begin{proof}
Let $V$ be a conical neighborhood of $p$ in $X$.
Choose $q' \in [p,q]_Y \setminus \set{p}$ and $r' \in [p,r]_Y \setminus \set{p}$ such that both $q',r' \in V$.
Since $p \notin [q,r]_Y$, by uniqueness of geodesics $p \notin [q',r']_Y$; thus we may assume the geodesic triangle $\widehat{T}$ in $Y$ with vertices $p,q',r'$ is simple.
%Notice $\widehat{T}$ lies completely in the convex hull of $\set{p,q',r'}$ in $X$ by \lemref{principal_lemma}, hence $\widehat{T} \subset V$.
By \corref{small_simple_limit_segments},
there exists $\epsilon > 0$ such that both limit segments $\widehat{R}_q^\epsilon$ and $\widehat{R}_r^\epsilon$ of $\widehat{T}$ exist at $p$, and both $\widehat{R}_q^\epsilon$ and $\widehat{R}_r^\epsilon$ lie completely in $\widehat{T} \cup \Int \widehat{T}$.
By \lemref{subtriangle_interiors}, we may assume %$\epsilon > 0$ is chosen small enough that
$\Int T \cap \overline{B}_X (p, \epsilon) = \Int \widehat{T} \cap \overline{B}_X (p, \epsilon)$; clearly we may also assume $T \cap \overline{B}_X (p, \epsilon) = \widehat{T} \cap \overline{B}_X (p, \epsilon)$.
Since the limit segments $R_q^\epsilon$ and $R_r^\epsilon$ of $T$ at $p$ coincide with the limit segments $\widehat{R}_q^\epsilon$ and $\widehat{R}_r^\epsilon$ of $\widehat{T}$ at $p$, %both $R_q^\epsilon$ and $R_r^\epsilon$ lie completely in $T \cup \Int T$.
we have proved the lemma.
\end{proof}

\begin{lemma}\label{feature_2}
Let $T$ be a simple geodesic triangle in $Y$ with vertices $p,q,r$.
For small enough $\epsilon > 0$, the limit segment $R_q^\epsilon$ of $T$ at $p$ can be extended to a geodesic in $X$ from $p$ to $q' \in [q,r]_Y \subset T$, with $[p,q']_X \subset T \cup \Int T$.
Similarly, the limit segment $R_r^\epsilon$ can be extended to a geodesic in $X$ from $p$ to $r' \in [q,r]_Y \subset T$, with $[p,r']_X \subset T \cup \Int T$.
\end{lemma}

\begin{proof}
Choose $\epsilon > 0$ small enough that both limit segments $R_q^\epsilon$ and $R_r^\epsilon$ of $T$ at $p$ lie completely in $T \cup \Int T$ (possible by \lemref{containing_limit_segments}).
If $R_q^\epsilon$ enters $\Int T$, then by \lemref{curve_theorem_2} we can extend $R_q^\epsilon$ to a geodesic $[p,q']_X$ in $X$ that lies completely in $T \cup \Int T$ with $q' \in T$; by \lemref{prop3}, $q' \in [q,r]_Y$, as desired.
So assume $R_q^\epsilon \subset T$.

Assume first that $R_q^\epsilon \subset [p,q]_Y$.
We may assume $R_q^\epsilon$ and $R_r^\epsilon$ lie in a conical neighborhood $V$ of $p$ in $X$.
Let $w$ be the midpoint of $R_q^\epsilon$; by \lemref{curve_theorem_1}, we may find be a sequence of points $w_k \in \Int T$ such that $w_k \to w$ in $X$ and each $w_k$ lies on a face of $X$.
For each $k \in \N$, let $[p_k,w_k]_X$ be the maximal subarc of the geodesic $[p,w_k]_X$ in $X$ between $p$ and $w_k$ such that $[p_k,w_k]_X \setminus \set{p_k} \subset \Int T$.
By \lemref{curve_theorem_2}, we can extend $[p_k,w_k]_X$ to a geodesic $[p_k,x_k]_X$ in $X$ such that $x_k \in T$ and $[p_k,x_k]_X \setminus \set{p_k,x_k} \subset \Int T$.
Because $p$ and $w$ lie in the conical neighborhood $V$, for all sufficiently large $k$ we find $[p,x_k]_X \setminus \set{p}$ lies on a face of $X$, and thus $p_k \in T$.
Since $T$ is simple and $[p,q]_Y$ follows $R_q^\epsilon$, which is a geodesic in $X$ with one endpoint $p$ and midpoint $w$, we see that eventually every $p_k \in [p,r]_Y$ (possibly $p_k = p$) and $p_k \to p$.
Since $w_k \in \Int T$ lies on $[p_k,x_k]_X$, by \lemref{prop3} we must have $x_k \in [p,q]_Y$ or $x_k \in [q,r]_Y$.
In the first case, we have a new simple geodesic triangle in $Y$ with vertices $p,p_k,x_k$ formed by concatenating the geodesics $[p,p_k]_Y \subset [p,r]_Y \subset T$, $[p_k,x_k]_X \subset E$, and $[x_k,p]_Y \subset [q,p]_Y \subset T$.
But this simple geodesic triangle has all its vertices along the (closed, convex) geodesic $[p,x_k]_X$ in $X$; by \lemref{principal_lemma} the whole triangle must lie in $[p,x_k]_X$, making it degenerate and therefore not simple.
Thus only the second case is possible:  $x_k \in [q,r]_Y$.
Taking the limit of geodesics $[p_k,x_k]_X$ in $X$ as $k \to \infty$, we obtain a geodesic $[p,x]_X$ in $X$; by compactness, $x \in [q,r]_Y$ and $[p,x]_X \subset T \cup \Int T$.
Thus we have proved the lemma in this case, too.

Finally, assume that $R_q^\epsilon$ is not a subset of $[p,q]_Y$.
Since $R_q^\epsilon \subset T$, we must have $R_q^\epsilon \subset [p,r]_Y$.
But then $R_r^\epsilon = R_q^\epsilon \subset [p,r]_Y$.
So, swapping the roles of $q$ and $r$ in the previous paragraph, we can extend $R_r^\epsilon = R_q^\epsilon$ in $T \cup \Int T$ to $[q,r]_Y$, as desired.

The proof for the limit segment $R_r^\epsilon$ is completely similar.
\end{proof}

\section{Curvature}

We finally arrive at the main result of the paper.

\begin{theorem}\label{main_theorem}
Let $E$ be a closed, rectifiably-connected subspace of a locally-finite {\rm CAT(\k)} $M_{\k}$-polyhedral $2$-complex $X$, where $\k \le 0$.
If $H_1(E) = 0$ then $E$, under the induced path metric, is a complete {\rm CAT(\k)} space.
\end{theorem}

\begin{proof}
Let $Y$ be the space $E$, with the induced path metric.
We want to show that 
$\angle_p^Y (q,r) \le \angle_p^{(\k), Y} (q,r)$ for every triple of distinct points $p,q,r \in E$.
So let $T$ be a geodesic triangle in $Y$ with distinct vertices $p,q,r \in E$.
If $T$ is not simple in a neighborhood of $p$ in $X$, then by uniqueness of geodesics, either $p$ lies on the geodesic $[q,r]_Y$ in $Y$ or both geodesics $[p,q]_Y$ and $[p,r]_Y$ in $Y$ must coincide for some distance from $p$.
In the first case, $T$ is degenerate with $\angle_p^Y (q,r) = \angle_p^{(\k), Y} (q,r) = \pi$.
In the second, $\angle_p^Y (q,r) = 0 \le \angle_p^{(\k), Y} (q,r)$.
Thus we may assume $T$ is simple in a neighborhood $U$ of $p$.

Let $[q,q']_Y$ be the maximal subarc (possibly degenerate) of the geodesic $[q,r]_Y$ in $Y$ such that $[q,q']_Y$ coincides with the geodesic $[q,p]_Y \subset T$ in $Y$ from $q$ to $p$.
The geodesic triangle in $Y$ with vertices $p,q,q'$ is therefore degenerate with $\angle_p^Y (q, q') = \angle_p^{(\k), Y} (q, q') = 0$.
By Alexandrov's Lemma \cite[Lemma II.4.10]{bridson}, it suffices to prove the theorem when $q = q'$.
Similarly, we may assume the geodesics $[r,q]_Y$ and $[r,p]_Y$ in $Y$ do not coincide near $r$.
Thus we may assume $T$ is simple.

By \thmref{limit_segment_theorem} and \lemref{feature_2}, we have a geodesic $[p,q'']_X$ in $X$ from $p$ to $q'' \in [q,r]_Y$ such that $[p,q'']_X \subset T \cup \Int T$.
The geodesic triangle in $Y$ with vertices $p,q,q''$ is either simple, in which case the equation in \thmref{limit_segment_theorem} applies, or both edges of the triangle from $p$ coincide for some positive distance out from $p$; in either case, $\angle_p^Y (q, q'') = 0 \le \angle_p^{(\k), Y} (q, q'')$.
By Alexandrov's Lemma \cite[Lemma II.4.10]{bridson}, it suffices to consider the case $q = q''$.
A similar argument holds for $r$ in place of $q$.
Thus we may assume both geodesics $[p,q]_Y$ and $[p,r]_Y$ in $Y$ lie completely in $X$.

Now $d (p, q') = d_Y (p, q')$ and $d (p, r') = d_Y (p, r')$.
And since $d (q', r') \le d_Y (q', r')$, we see directly that
\[\angle_p^Y (q', r') = \angle_p^X (q', r') \le \angle_p^{(\k), X} (q', r') \le \angle_p^{(\k), Y} (q', r').\]
(Recall $\angle_p^X (q',r') \le \angle_p^{(\k), X} (q',r')$ by the CAT(\k) inequality.)
Thus $Y$ is CAT(\k).
Completeness comes from \lemref{complete and geodesic}.
\end{proof}

\bibliographystyle{amsplain}
\bibliography{catk2cxs}

\end{document}